\documentclass[12pt]{article}

\usepackage[utf8]{inputenc} 
\usepackage[T1]{fontenc}
\usepackage[english]{babel}

\usepackage{microtype} 
\usepackage{lmodern}

\usepackage{amsmath, amssymb, amsthm, mathtools}

\usepackage{graphicx}
\usepackage{wrapfig}
\usepackage{caption}
\usepackage{subcaption}
\usepackage{tikz}
\usetikzlibrary{decorations.pathreplacing,arrows.meta,positioning}

\usepackage[ruled,vlined]{algorithm2e}

\usepackage{enumitem}
\usepackage{multicol}
\usepackage[numbers,sort&compress]{natbib}

\usepackage[hidelinks]{hyperref}

\theoremstyle{plain}

\newtheorem{theorem}{Theorem}
\numberwithin{theorem}{section}

\newtheorem{conjecture}{Conjecture}
\numberwithin{conjecture}{section}

\newtheorem{corollary}{Corollary}
\numberwithin{corollary}{section}

\numberwithin{example}{section}

\newtheorem{lemma}{Lemma}
\numberwithin{lemma}{section}

\numberwithin{proposition}{section}

\numberwithin{problem}{section}

\newtheorem{remark}{Remark}
\numberwithin{remark}{section}

\newtheorem{claim}{Claim}
\numberwithin{claim}{theorem}

\newtheorem{definition}{Definition}
\numberwithin{definition}{section}

\author{
Hadeel Al Bazzal\thanks{KALMA, Faculty of Sciences, Lebanese University, Baalbek, Lebanon; LIB, Universit\'e Bourgogne Europe, Dijon, France}}

\usepackage{fullpage}
\begin{document}

\title{$t$-tone edge coloring of graphs}
\maketitle

\begin{abstract}
In this paper, we introduce the notion of $t$-tone edge coloring. A $t$-tone edge $k$-coloring of a graph $G$ assigns to each edge of $G$ a set of $t$ distinct colors from $\{1,\dots,k\}$ such that any two edges at distance $d$ share fewer than $d$ common colors. The $t$-tone chromatic index of $G$, denoted by $\tau'_t(G)$, is the minimum integer $k$ for which $G$ admits a $t$-tone edge $k$-coloring. We focus on the case $t=2$ and establish several upper bounds on $\tau'_2$. In particular, for every graph $G$ with maximum degree $\Delta(G)\ge2$, we prove that $\tau'_2(G)\le 6\Delta(G)-4$, improving the corresponding bound derived from the vertex analogue. We also show that every tree $T$ with $\Delta(T)\ge3$ satisfies $\tau'_2(T)=2\Delta(T)$. Furthermore, every planar graph $G$ satisfies $\tau'_2(G)\le \max\{41,3\Delta(G)+5\}$, while every outerplanar graph $G$ satisfies $\tau'_2(G)\le \max\{14,3\Delta(G)\}$. For subcubic graphs $G$, the vertex analogue yields $\tau'_2(G)\le12$. We improve this bound to $11$ for claw-free subcubic graphs and to $10$ for $2$-degenerate subcubic graphs. Finally, we propose two conjectures concerning optimal bounds for cubic and $K_4$-free cubic graphs, and establish them for series-parallel subcubic multigraphs and subcubic outerplanar graphs, respectively.
\end{abstract}



\section{Introduction}

Throughout this paper, all graphs are assumed to be simple and finite unless stated otherwise. We denote the vertex set and edge set of a graph $G$ by $V(G)$ and $E(G)$, respectively. In a proper coloring of a graph $G$, a color (integer) is assigned to each vertex of $G$ such that adjacent vertices receive distinct colors. The chromatic number of $G$, denoted by $\chi(G)$, is the minimum integer $k$ for which $G$ admits a proper coloring. Over time, this classical notion has inspired numerous generalizations. For instance, in a distance-$t$ coloring \cite{KS}, the condition is strengthened so that any two vertices at distance at most $t$ receive distinct colors. Another extension is the notion of a $t$-tone $k$-coloring \cite{BP, F}. A $t$-tone $k$-coloring of a graph $G$ assigns to each vertex of $G$ a set of $t$ distinct colors from $\{1,\dots,k\}$ such that any two vertices at distance $d$ share fewer than $d$ common colors. The minimum integer $k$ for which $G$ admits a $t$-tone $k$-coloring is called the $t$-tone chromatic number of $G$, denoted by $\tau_t(G)$. Note that the distance between two distinct vertices $u$ and $v$ in $G$, denoted by $d_G(u,v)$, is the length of a shortest path between $u$ and $v$ in $G$.

Edge analogues of several vertex coloring notions have also been studied. In particular, a proper edge coloring of a graph $G$ assigns colors to each edge of $G$ such that adjacent edges receive distinct colors. The chromatic index of $G$, denoted by $\chi'(G)$, is the minimum integer $k$ for which $G$ admits a proper edge coloring. Furthermore, a distance-$t$ edge coloring \cite{Kang} of a graph $G$ assigns a color to each edge of $G$ such that any two edges at distance at most $t$ receive distinct colors. The distance-$t$ chromatic index of $G$, denoted by $\chi'_t(G)$, is the minimum integer $k$ for which $G$ admits a distance-$t$ edge coloring. This parameter is related to graph powers, since $\chi'_t(G)=\chi((L(G))^t)$, where $L(G)$ denotes the line graph of $G$, and $G^t$ denotes the $t$-th power of a graph $G$, obtained by joining all pairs of vertices at distance at most $t$. In the context of distance-$t$ edge coloring, the distance between two edges $e,e'\in E(G)$ is defined as the distance between the corresponding vertices in the line graph $L(G)$ of $G$. We adopt the same definition throughout this paper and denote this distance by $d_G(e,e')$.

While $t$-tone coloring has been investigated for vertices, to the best of our knowledge, no edge analogue of this notion has yet been studied. In this paper, we introduce and study the notion of $t$-tone edge coloring.

\begin{definition}
Let $G$ be a graph and let $k,t \in \mathbb{N}$ with $1 \le t \le k$. Let $\binom{\{1,\dots,k\}}{t}$ denote the set of all $t$-element subsets of $\{1,\ldots,k\}$. A \emph{$t$-tone edge $k$-coloring} of $G$ is a mapping $f:E(G)\to \binom{\{1,\dots,k\}}{t}$ such that for any two edges $e,e'\in E(G)$, we have $|f(e)\cap f(e')|< d_G(e,e')$.

\noindent A graph admitting a $t$-tone edge $k$-coloring is said to be $t$-tone edge $k$-colorable. The $t$-tone chromatic index of $G$, denoted by $\tau'_t(G)$, is the minimum integer $k$ such that $G$ admits a $t$-tone edge $k$-coloring.
\end{definition}

Given a $t$-tone edge coloring $f$ of $G$, we call $f(e)$ the label of $e$ and the elements of $\{1,\dots,k\}$ colors. When no ambiguity arises, we omit set brackets in labels; that is, the label $\{a,b\}$ is written simply as $ab$. Edges $e$ and $e'$ (or vertices $x$ and $y$) are considered \textit{neighbors} when their distance is one, and \textit{second neighbors} when their distance is two.

The degree of a vertex $v$ in $G$, denoted by $d_G(v)$, is the number of edges incident with $v$. The maximum vertex degree of $G$ is denoted by $\Delta(G)$. Similarly, the degree of an edge $e$ of $G$, denoted by $d_G(e)$, is the number of edges adjacent to $e$. In particular, if $e=xy$, then $d_G(e)=d_G(x)+d_G(y)-2$. For an edge $xy\in E(G)$, we also define
\[
\eta_G(xy)
=
\left|\bigl(N_G(x)\cup N_G(y)\bigr)\setminus\{x,y\}\right|
=
|N_G(x)\cup N_G(y)|-2.
\] 
Graph families are denoted by $C_n$, $P_n$, and $K_n$ for the cycle, path, and complete graph on $n$ vertices, respectively, and by $K_{1,n}$ for the star on $n+1$ vertices. Clearly, $\tau'_t(K_{1,n})= tn$.

The notion of $t$-tone edge coloring generalizes proper edge coloring. In particular, when $t=1$, we obtain $\tau'_1(G)=\chi'(G)$ for every graph $G$. Note that, for each $t$, the parameter $\tau'_t$ is monotone under taking subgraphs: if $H$ is a subgraph of $G$, then $\tau'_t(H)\le \tau'_t(G)$. A natural lower bound for $\tau'_t(G)$ arises from the local obstruction $K_{1,\Delta(G)}$. Consequently, for every graph $G$, $\tau'_t(G) \ge t\,\Delta(G)$.

Since a $t$-tone edge coloring of a graph $G$ corresponds precisely to a $t$-tone coloring of its line graph $L(G)$, several results on $t$-tone coloring can be transferred directly to the edge-setting. In particular, using $\tau'_t(G)=\tau_t(L(G))$, we obtain the following bounds for $\tau'_t(G)$. Bickle and Phillips \cite{BP} proved that $\tau_t(G) \ge \tau_{t-1}(G) + 2$ for every nonempty graph $G$ and for $t \ge 2$. Applying this result to $L(G)$ immediately yields the following lower bound.

\begin{corollary}
Let $G$ be a nonempty graph and $t \ge 2$. Then $\tau'_t(G) \ge \tau'_{t-1}(G) + 2$.  
\end{corollary}

Fonger et al.\ \cite{F} established that $\tau_t(G) \le \tau_{t-1}(G) + \chi(G^t)$ for $t \ge 2$. This allows us to obtain the following upper bound on $\tau'_t(G)$.

\begin{corollary}
Let $G$ be a graph and $t \ge 2$. Then $\tau'_t(G) \le \tau'_{t-1}(G) + \chi'_t(G)$.
\end{corollary}

Another general upper bound follows from a result established in \cite{C1}, which gives the best known upper bound for $\tau_t(G)$. It states that for every integer $t$ and every nonempty graph $G$, $\tau_t(G) \le (t^2 + t)\Delta(G)$. Applying this inequality to the line graph $L(G)$ and using the fact that $\Delta(L(G)) \le 2\Delta(G) - 2$, we obtain the following upper bound.

\begin{corollary}
For every integer $t$ and every nonempty graph $G$, $\tau'_t(G) \le (t^2 + t)(2\Delta(G) - 2)$.
\end{corollary}

The $t$-tone chromatic number is known exactly for all values of $t$ in the case of paths. Bickle and Phillips \cite{BP} proved that $\tau_t(P_n) =\sum_{i=0}^{n-1} \max \Bigl\{ 0,\, t - \binom{i}{2} \Bigr\}$ for every integer $n\ge1$. Since $L(P_n)=P_{n-1}$, we immediately obtain the following corollary.

\begin{corollary}\label{co1}
Let $P_n$ be a path of order $n\ge 2$. Then, \[\tau'_t(P_n) = \sum_{i=0}^{n-2} \max \Bigl\{ 0,\, t - \binom{i}{2} \Bigr\}.\]
\end{corollary}

\section{2-tone edge coloring of graphs}

In this section, we study the case $t=2$. In this case, a $2$-tone edge coloring of a graph $G$ assigns to each edge of $G$ a set of two colors such that adjacent edges receive disjoint sets and edges at distance two receive distinct sets.

A natural starting point is the class of cycles. In \cite{BP}, it was shown that $\tau_2(C_n)$ equals six when $n\in \{3,4,7\}$, and equals five otherwise. Since $L(C_n)=C_n$, we immediately obtain the following corollary.

\begin{corollary}\label{cc1}
For the cycle $C_n$, 
\[
\tau'_2(C_n)=
\begin{cases}
6, & \text{if } n \in \{3,4,7\}\\
5, & \text{otherwise}
\end{cases}.
\]
\end{corollary}

For a graph $G$, we have $\tau'_2(G)=0$ if $G$ is empty, and $\tau'_2(G)=2$ when $\Delta(G)=1$. If $\Delta(G)=2$, then $G$ is a disjoint union of paths and cycles, and hence $\tau'_2(G)\in\{5,6\}$ by Corollaries \ref{co1} and \ref{cc1}. Therefore, it suffices to consider graphs with $\Delta(G)\ge 3$.

The primary approach in this paper involves constructing a $t$-tone edge $k$-coloring of a graph iteratively. 

\begin{definition}
A \emph{partial $t$-tone edge $k$-coloring} of a graph $G$ is a function 
$f : E' \to \binom{\{1,\dots,k\}}{t}$, where $E' \subseteq E(G)$, such that for every pair of edges $e,e' \in E'$, we have $|f(e)\cap f(e')| < d_G(e,e')$. Edges in $E(G) \setminus E'$ are called \emph{uncolored}.
\end{definition}

Note that a $t$-tone edge $k$-coloring of a subgraph $H \subseteq G$ is not necessarily a partial $t$-tone edge $k$-coloring of $G$, as the distance between a pair of edges in $G$ may be strictly less than their distance in $H$. This issue, however, does not arise when the subgraph is formed by vertex deletion.

\begin{remark}\label{r1}
 Let $G'$ be obtained from a graph $G$ by deleting a vertex. Since removing a vertex does not decrease the distance between any pair of edges in $G'$, any $t$-tone edge $k$-coloring of $G'$ is a partial $t$-tone edge $k$-coloring of $G$.
\end{remark}

We now introduce the following terminology.

\begin{definition}
Let $f$ be a partial $2$-tone edge coloring of a graph $G$, and let $e \in E(G)$ be an uncolored edge. A \emph{valid label} for $e$ is a label by which $f$ can be extended to $e$. A color is \emph{free at $e$} if it does not appear on any edge adjacent to $e$. Let $f_r(e)$ denote the number of free colors at $e$. A \emph{candidate label} for $e$ is a label consisting only of free colors. 
\end{definition}

The notions of partial coloring, valid label, free color, and candidate label were introduced by Cranston, Kim, and Kinnersley \cite{C1} for $t$-tone coloring; the definitions above are their edge analogues.

As mentioned in the introduction, every graph $G$ satisfies $\tau'_t(G)\ge t\Delta(G)$. In particular, for $t=2$, we obtain $\tau'_2(G)\ge 2\Delta(G)$. We now show that this bound is sharp for trees with maximum degree at least three.

\begin{theorem}
Let $T$ be a tree with $\Delta(T)\ge 3$. Then $\tau'_2(T)=2\Delta(T)$.
\end{theorem}

\begin{proof}
Let $\Delta=\Delta(T)$. Since $T$ contains $K_{1,\Delta}$ as a subgraph and $\tau'_2(K_{1,\Delta})=2\Delta$, the monotonicity of $\tau'_2$ under subgraphs implies that $\tau'_2(T)\ge 2\Delta$. It remains to show that $T$ admits a $2$-tone edge coloring using $2\Delta$ colors.

\noindent Suppose, for a contradiction, that this is false, and let $T$ be a counterexample with the minimum number of vertices. Then $T$ is not a star, and hence it contains at least two internal vertices. Fix a root $r$ with $d_T(r)=\Delta$, and let $v$ be a non-leaf vertex at maximum distance from $r$. Then every neighbor of $v$ other than its parent $p$ (possibly $p=r$) is a leaf. Let $v_1$ be one such leaf, and set $T' = T - \{v_1\}$. Since $r \in T'$, $\Delta(T')=\Delta$. By the minimality of $T$, the tree $T'$ admits a $2$-tone edge $2\Delta$-coloring $f$. By Remark \ref{r1}, $f$ is a partial $2$-tone edge coloring of $T$.

\noindent We claim that $d_T(v)\ge 3$. Otherwise, since $d_T(v)\ge 2$, we have $d_T(v)=2$, and thus $v_1$ and $p$ are the only neighbors of $v$ in $T$. Hence, $f_r(vv_1)\ge 2\Delta-2$. Thus, there are at least $\binom{2\Delta-2}{2}$ candidate labels for $vv_1$. The second neighbors of $vv_1$ are precisely the edges incident with $p$ other than $vp$, and there are at most $\Delta-1$ such edges. Since $\binom{2\Delta-2}{2} > \Delta-1$ for $\Delta \ge 3$, at least one candidate label is valid for $vv_1$, a contradiction.

\noindent Thus, $d_T(v)\ge 3$, and hence $v$ is adjacent to at least two leaves. Consequently, $f_r(vv_1)\ge2$. Let $a$ and $b$ be two colors free at $vv_1$. If the label $ab$ does not appear on any second neighbor of $vv_1$, then $ab$ is a valid label for $vv_1$, a contradiction. Therefore, the label $ab$ must appear on a second neighbor of $vv_1$, which implies that there exists an edge incident with $p$, distinct from $vp$, assigned $ab$. Let $v_2$ be a leaf neighbor of $v$ distinct from $v_1$, and let $f(vv_2)=cd$. Since the second neighbors of $vv_2$ are precisely the edges incident with $p$ other than $vp$, one of which is colored $ab$, we recolor $vv_2$ with $ca$, and assign $db$ to $vv_1$, a contradiction.\end{proof}

Before proceeding further, we introduce the following additional terminology.

\begin{definition}
Let $e=xy$ be an edge of a graph $G$. A vertex $u\in V(G)$ is an \emph{intermediate vertex} of $e$ if $u$ is adjacent to $x$ or $y$ and is incident with a second neighbor of $e$.
\end{definition}

\begin{definition}
Let $f$ be a partial $2$-tone edge coloring of a graph $G$, and let $e \in E(G)$ be an uncolored edge. A second neighbor $e'$ of $e$ is called a \emph{forbidding edge} for $e$ if $f(e')$ is a candidate label for $e$. In this case, we say that $e'$ \emph{forbids} a label for $e$.

\noindent For any subset of candidate labels $L$, we say that $e'$ is a \emph{forbidding edge for $e$ under $L$} if $f(e') \in L$. In this case, we say that $e'$ \emph{forbids a label from $L$} for $e$.
\end{definition}

\begin{remark}\label{r2}
Let $f$ be a partial $2$-tone edge coloring of a graph $G$, and let $e\in E(G)$ be an uncolored edge. If $L$ is a family of pairwise intersecting candidate labels for $e$, then each intermediate vertex of $e$ can be incident to at most one forbidding edge for $e$ under $L$. Otherwise, two such edges would be adjacent, requiring their labels to be disjoint, which contradicts the fact that $L$ is pairwise intersecting.
\end{remark}

With these definitions in place, our first step is to determine the maximum size of a family of pairwise intersecting candidate labels for a given edge.

\begin{lemma}\label{p1}
Let $f$ be a partial $2$-tone edge coloring of a graph $G$, and let $e$ be an uncolored edge of $G$ with $f_r(e) = m > 2$. Let $L$ be a largest family of pairwise intersecting candidate labels for $e$. Then,
\[
|L| =
\begin{cases}
m - 1, & \text{if } m > 3 \\
3, & \text{if } m = 3
\end{cases}.
\]
\end{lemma}

\begin{proof}
If $m = 3$, then the result follows since all $2$-element subsets of a $3$-element set are pairwise intersecting. Assume now that $m \ge 4$. By the Erd\H{o}s--Ko--Rado Theorem \cite{E}, which states that for $n \ge 2k$, the largest intersecting family of $k$-subsets of an $n$-element set has size $\binom{n-1}{k-1}$, setting $n = m$ and $k = 2$ yields $|L| = \binom{m-1}{1} = m - 1$.\end{proof}

A natural approach for extending a partial $2$-tone edge coloring of a graph $G$ to an uncolored edge $e\in E(G)$ is to ensure that the number of candidate labels for $e$ exceeds the number of its second neighbors. When this condition fails, we instead exploit families of pairwise intersecting labels, as formalized in the following lemma.

\begin{lemma}\label{l1}
Let $f$ be a partial $2$-tone edge coloring of a graph $G$, and let $e$ be an uncolored edge of $G$. Let $L$ be a set of pairwise intersecting candidate labels for $e$. If $|L| > m$, where $m$ is the number of intermediate vertices of $e$, then $f$ can be extended to $e$ in at least $|L| - m$ distinct ways.
\end{lemma}

\begin{proof}
Let $m$ be the number of intermediate vertices of $e$ in $G$. By Remark \ref{r2}, each such vertex is incident to at most one forbidding edge for $e$ under $L$; thus, $e$ has at most $m$ forbidding edges under $L$. Since each forbidding edge corresponds exactly to one label in $L$, at least $|L| - m$ labels in $L$ are valid for $e$. Each such label yields a distinct extension of $f$ to $e$.\end{proof}

Since the number of intermediate vertices of $e$ is at most $\eta_G(e)$, Lemma \ref{l1} immediately yields the following corollary.

\begin{corollary}\label{x}
Let $f$ be a partial $2$-tone edge coloring of a graph $G$, and let $e$ be an uncolored edge of $G$. Let $L$ be a set of pairwise intersecting candidate labels for $e$. If $|L| > \eta_G(e)$, then $f$ can be extended to $e$ in at least $|L| - \eta_G(e)$ distinct ways.
\end{corollary}

To establish a baseline for our main result, we recall the best known upper bound on the $2$-tone chromatic number. It was shown in \cite{B1} that every graph $G$ with \(\Delta(G)\ge2\) satisfies
\[
\tau_2(G)\le 2\Delta(G)-1+\left\lceil \frac{1+\sqrt{1+8\Delta(G)(\Delta(G)-1)}}{2}\right\rceil .
\]
Applying this bound to the line graph $L(G)$ yields the following corollary.

\begin{corollary}\label{cc2}
For every graph $G$ with maximum degree $\Delta \ge 2$,
\[
\tau'_2(G) \le 4\Delta - 5 + \left\lceil \frac{1 + \sqrt{1 + 8(2\Delta - 2)(2\Delta - 3)}}{2} \right\rceil.
\]
\end{corollary}

We improve upon this result by providing the following upper bound.

\begin{theorem}\label{t1}
For every graph $G$ with $\Delta(G) \ge 2$, $\tau'_2(G) \le 6\Delta(G) - 4$.
\end{theorem}

\begin{proof}
Let $G$ be a graph with maximum degree $\Delta = \Delta(G)$. Let $k=6\Delta-4$ and let $E(G)=\{e_1,e_2,\dots,e_n\}$. Starting with all edges uncolored, we extend our partial edge coloring to $e_1,e_2,\dots,e_n$ in order. Consider an edge $e_i$ with $i\in\{1,\dots,n\}$. Since $d_G(e_i)\le 2\Delta-2$, it follows that $f_r(e_i)\ge 2\Delta$. Let $L$ be a largest family of pairwise intersecting candidate labels for $e_i$. Since $\Delta\ge2$, we have $f_r(e_i)\ge4$. Hence, Lemma~\ref{p1} implies that $|L|\ge2\Delta-1$. Since $\eta_G(e_i) \le 2\Delta -2$, we have $|L|>\eta_G(e_i)$. Therefore, by Corollary \ref{x}, at least one label in $L$ is valid for $e_i$, and the coloring can be extended to $e_i$. Proceeding iteratively, we obtain a $2$-tone edge $k$-coloring of $G$.\end{proof}

Theorem~\ref{t1} improves the bound of Corollary~\ref{cc2} whenever $\Delta(G)>4$. Asymptotically, our result reduces the leading constant from approximately $4 + 2\sqrt{2} \approx 6.83$ to $6$, yielding an improvement of about $0.83\Delta$ colors. When $\Delta(G) = 4$, both results yield the same bound, namely $20$. For $\Delta(G) \le 3$, however, Corollary \ref{cc2} provides a sharper bound ($13$) than the one obtained here ($14$).

We now turn to planar and outerplanar graphs. Our proofs rely on the minimal counterexample method together with the structural properties of planar and outerplanar graphs provided by the following lemmas.

\begin{lemma}[\cite{B}]\label{l2}
For every planar graph $G$ there exists $u \in V(G)$ such that $d_G(u) \le 5$ and $u$ has at most two neighbors with degree at least $11$.
\end{lemma}

\begin{lemma}[\cite{FA}]\label{l3'}
For every outerplanar graph $G$ there exists $xy \in E(G)$ with $d_G(x) = 1$, or $d_G(x) = 2$ and $d_G(y) \le 4$.
\end{lemma}

\begin{theorem}
For every planar graph $G$, $\tau_2'(G) \le \max\{41,\,3\Delta(G)+5\}$.
\end{theorem}

\begin{proof}
Assume, for a contradiction, that the statement is false, and let $G$ be a counterexample with minimum order. Set $\Delta=\Delta(G)$ and $k=\max\{41,\,3\Delta+5\}$. 

\noindent By Lemma \ref{l2}, there exists $u\in V(G)$ with $d_G(u)\le 5$ such that at most two of its neighbors have degree at least $11$. Let $G'=G-u$. By the minimality of $G$, $G'$ admits a $2$-tone edge $k$-coloring $f$. By Remark \ref{r1}, $f$ is a partial $2$-tone edge $k$-coloring of $G$. Let $u_1,\dots,u_{d_G(u)}$ be the neighbors of $u$ in $G$, and assume, without loss of generality, that $u_1$ and $u_2$ are the neighbors of degree at least $11$, if such neighbors exist. We extend $f$ to the edges $uu_1,\dots,uu_{d_G(u)}$ in this order. For each $i$, let $e_i=uu_i$.

\noindent For $i\in\{1,2\}$, we have $d_G(u_i)\le \Delta$ and at most $i-1\le 1$ edges at $u$ have been colored; hence, $f_r(e_i)\ge k-2((\Delta-1)+(i-1))\ge k-2\Delta$. Let $L$ be a largest family of pairwise intersecting candidate labels for $e_i$. Since $f_r(e_i) > 3$, Lemma \ref{p1} implies that $|L| \ge k-2\Delta-1$. We have $\eta_G(e_i) \le \Delta +3$; hence, $|L|>\eta_G(e_i)$. Thus, by Corollary \ref{x}, the coloring extends to $e_i$.

\noindent Now for $i\ge 3$, we have $d_G(u_i)\le 10$, so at most $9$ edges at $u_i$ are colored, and at most $i-1\le 4$ edges at $u$ are colored. Thus $f_r(e_i)\ge k-2(9+4)=k-26$. Let $L'$ be a largest family of pairwise intersecting candidate labels for $e_i$. Since $f_r(e_i)> 3$, Lemma \ref{p1} implies that $|L'| \ge k-27$. We have $\eta_G(e_i)\le 13$; hence, $|L'|>\eta_G(e_i)$. Therefore, by Corollary \ref{x}, the coloring extends to $e_i$. This yields a $2$-tone edge $k$-coloring of $G$, a contradiction.\end{proof}

While the previous result applies to all planar graphs, a sharper bound can be derived for outerplanar graphs.

\begin{theorem}\label{t4}
If $G$ is an outerplanar graph, then $\tau_2'(G)\le \max\{14,\,3\Delta(G)\}$.
\end{theorem}

\begin{proof}
Assume, for a contradiction, that the statement is false, and let $G$ be a counterexample with minimum order. Set $\Delta=\Delta(G)$ and $k=\max\{14,\,3\Delta\}$. By Lemma \ref{l3'}, there exists an edge $uv \in E(G)$ such that $d_G(u)=1$, or $d_G(u)=2$ and $d_G(v)\le 4$. Let $G'=G-u$. By the minimality of $G$, the graph $G'$ admits a $2$-tone edge $k$-coloring $f$, which is, by Remark \ref{r1}, a partial $2$-tone edge coloring of $G$.

\noindent If $d_G(u)=1$, then at most $\Delta-1$ edges adjacent to $uv$ are colored. Hence, $f_r(uv)\ge k-2(\Delta-1)$. Let $L$ be a largest family of pairwise intersecting candidate labels for $uv$. Since $f_r(uv) > 3$, Lemma \ref{p1} implies that $|L| \ge k-2(\Delta-1)-1\ge \Delta +1$. Since $\eta_G(uv)\le \Delta-1$, it follows that $|L|>\eta_G(uv)$, and the coloring extends to $uv$ by Corollary \ref{x}, a contradiction.

\noindent Hence, we may assume that $d_G(u)=2$. Let $w$ be its neighbor other than $v$. We first extend $f$ to the edge $uw$. At this stage, no edge incident to $u$ is colored, and at most $\Delta-1$ edges incident to $w$ are colored. Hence, $f_r(uw)\ge k-2(\Delta-1)$. Let $L'$ be a largest family of pairwise intersecting candidate labels for $uw$. Since $f_r(uw) > 3$, Lemma \ref{p1} implies that $|L'| \ge k-2(\Delta-1)-1\ge \Delta+1$. Since $\eta_G(uw)\le \Delta<|L'|$, the coloring extends to $uw$ by Corollary \ref{x}.

\noindent Now consider $uv$. At most $3$ edges incident to $v$ are already colored, and one edge incident to $u$ is colored. Thus, $f_r(uv)\ge k-8$. Let $L''$ be a largest family of pairwise intersecting candidate labels for $uv$. Since $f_r(uv) > 3$, Lemma \ref{p1} implies that $|L''| \ge k-9$. Since $\eta_G(uv)\le 4$, we have $|L''|>\eta_G(uv)$, and the coloring extends to $uv$ by Corollary \ref{x}, a contradiction.\end{proof}

We now turn to subcubic graphs, that is, graphs with maximum degree at most three. Graphs with maximum degree exactly three are called cubic. The general lower bound on $\tau'_2$ yields $\tau'_2(G)\ge6$ for every subcubic graph $G$. The Petersen graph shows that this bound is sharp by admitting a $2$-tone edge $6$-coloring (see Figure~\ref{f1}). Furthermore, a direct verification shows that $\tau'_2(K_4-e)=8$ and $\tau'_2(K_4)=9$.

\begin{figure}[htbp]
\centering
\begin{tikzpicture}[
scale=1.6,
every node/.style={circle, draw, inner sep=1.5pt},
lab/.style={draw=none, fill=white, inner sep=0.8pt, font=\scriptsize}
]

\node (v1) at (90:2) {};
\node (v2) at (18:2) {};
\node (v3) at (-54:2) {};
\node (v4) at (-126:2) {};
\node (v5) at (162:2) {};

\node (u1) at (90:1.1) {};
\node (u2) at (18:1.1) {};
\node (u3) at (-54:1.1) {};
\node (u4) at (-126:1.1) {};
\node (u5) at (162:1.1) {};

\draw[very thick] (v1)--(v2) node[midway, lab] {23};
\draw[very thick] (v2)--(v3) node[midway, lab] {45};
\draw[very thick] (v3)--(v4) node[midway, lab] {13};
\draw[very thick] (v4)--(v5) node[midway, lab] {25};
\draw[very thick] (v5)--(v1) node[midway, lab] {14};

\draw[very thick] (u1)--(u3) node[midway, lab] {34};
\draw[very thick] (u3)--(u5) node[midway, lab] {15};
\draw[very thick] (u5)--(u2) node[midway, lab] {24};
\draw[very thick] (u2)--(u4) node[midway, lab] {35};
\draw[very thick] (u4)--(u1) node[midway, lab] {12};

\draw[very thick] (v1)--(u1) node[midway, lab] {56};
\draw[very thick] (v2)--(u2) node[midway, lab] {16};
\draw[very thick] (v3)--(u3) node[midway, lab] {26};
\draw[very thick] (v4)--(u4) node[midway, lab] {46};
\draw[very thick] (v5)--(u5) node[midway, lab] {36};

\end{tikzpicture}
\caption{A $2$-tone edge $6$-coloring of the Petersen graph.}
\label{f1}
\end{figure}
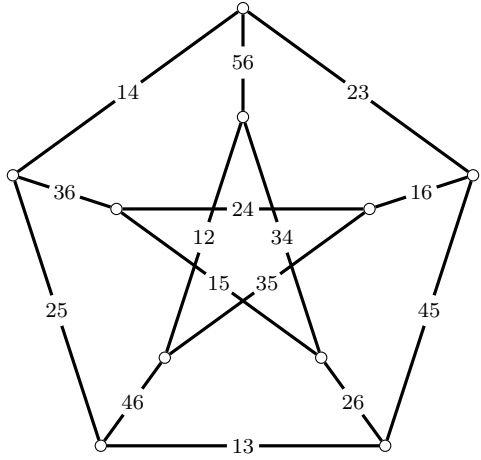

By a result of Dong \cite{Dong1}, every graph $H$ with $\Delta(H)\le4$ satisfies $\tau_2(H)\le12$. Since $\Delta(L(G))\le4$ for every subcubic graph $G$, it follows that $\tau'_2(G)\le12$, improving the bound of Corollary~\ref{cc2} from $13$ to $12$. We strengthen this result by further reducing the bound to $11$ for claw-free subcubic graphs and to $10$ for $2$-degenerate subcubic graphs, where a graph is said to be \emph{claw-free} if it contains no induced subgraph isomorphic to $K_{1,3}$ (the claw), and \emph{$2$-degenerate} if every subgraph contains a vertex of degree at most two.

To establish these bounds, we study the extension of partial colorings after vertex deletion. This leads to several local configurations, which are handled in the following lemmas. Before stating the lemmas, we introduce the following definition.

\begin{definition}
Let $e=xy$ be an edge of a graph $G$. The edges incident with $x$, distinct from $e$, are called the \emph{$yx$-edges}. Similarly, the edges incident with $y$, distinct from $e$, are called the \emph{$xy$-edges}.
\end{definition}

\begin{lemma}\label{l3}
Let $f$ be a partial $2$-tone edge $7$-coloring of a subcubic graph $G$, and let $e$ be an uncolored edge of $G$. If one end of $e$ has degree one in $G$, then $f$ can be extended to $e$.
\end{lemma}

\begin{proof}
Let $e=uv$, and assume, without loss of generality, that $d_G(u)=1$. Since $v$ is incident with at most two edges distinct from $e$, at least three colors are free at $e$. Consequently, $e$ admits three pairwise intersecting candidate labels. Since $\eta_G(e)\le 2$, Corollary~\ref{x} ensures that $f$ extends to $e$.\end{proof}

\begin{lemma}\label{l4}
Let $f$ be a partial $2$-tone edge $10$-coloring of a subcubic graph $G$. If $G$ contains a vertex $u$ of degree two whose incident edges are both uncolored, then $f$ can be extended to these edges.
\end{lemma}

\begin{proof}
Let $uu_1$ and $uu_2$ be the two uncolored edges incident with $u$. We first color the edge $uu_1$. Since $uu_2$ is uncolored and $u_1$ is incident with at most two colored edges, at least six colors are free at $uu_1$. Denote these colors by $1,2,3,4,5,6$ and consider $L=\{12,13,14,15,16\}$. By Corollary \ref{x}, since $\eta_G(uu_1)\le 3$, at least two labels in $L$ are valid for $uu_1$; assume these are $15$ and $16$.

\noindent If one of these labels shares a color with $uu_2$-edges, then at least five colors are free at $uu_2$, yielding at least ten candidate labels. Since $uu_2$ has at most six second neighbors, at least one of these labels is valid for $uu_2$.

\noindent Otherwise, neither $15$ nor $16$ intersects the colors on $uu_2$-edges. Assign $15$ to $uu_1$, so $6$ is free at $uu_2$. At most six colors appear on edges adjacent to $uu_2$, hence at least four colors are free at $uu_2$; denote them by $a,b,c,6$. Consider $L'=\{a6,b6,c6\}$. Since $6$ is free at $uu_1$, no $uu_1$-edge forbids labels from $L'$ for $uu_2$. Moreover, $uu_2$ has at most two intermediate vertices distinct from $u_1$, and thus, by Lemma \ref{l1}, $uu_2$ has at most two forbidding edges under $L'$. Hence, some label in $L'$ is valid for $uu_2$, and $f$ extends.\end{proof}

\begin{lemma}\label{l5}
Let $f$ be a partial $2$-tone edge $9$-coloring of a cubic graph $G$. Let $u\in V(G)$ with neighbors $u_1, u_2, u_3$ such that $u_2$ is adjacent to both $u_1$ and $u_3$. If the edges $uu_1, uu_2,$ and $uu_3$ are uncolored, then $f$ can be extended to these edges.
\end{lemma}

\begin{proof}
Let $u_1'$ and $u_3'$ be the neighbors of $u_1$ and $u_3$ outside the diamond, respectively. Suppose, without loss of generality, that $f(u_1u_2)=12$ and $f(u_2u_3)=34$.

\noindent We first color the edge $uu_1$. Since the distance between $u_2u_3$ and $u_1u_1'$ is at most two, $f(u_1u_1')\ne 34$. Thus, at least one of the colors $3$ or $4$ is free at $uu_1$. Without loss of generality, we may assume that $3$ is free at $uu_1$. Moreover, at least five colors are free at $uu_1$. Pairing $3$ with the other free colors at $uu_1$ yields a pairwise intersecting family $L$ of four candidate labels for $uu_1$. Since $uu_1$ has at most two intermediate vertices, it follows from Lemma \ref{l1} that at least one label in $L$ is valid for $uu_1$.

\noindent Next, we color the edge $uu_3$. The distance between $u_1u_2$ and $u_3u_3'$ is at most two, so $f(u_3u_3')\ne 12$. Hence, one of the colors $1$ or $2$ is free at $uu_3$; assume, without loss of generality, that $1$ is free at $uu_3$. Furthermore, at least four colors are free at $uu_3$, including $1$. Pairing $1$ with the other free colors at $uu_3$ yields a pairwise intersecting family $L'$ of three candidate labels for $uu_3$. Since $uu_3$ has at most two intermediate vertices, Lemma \ref{l1} implies that at least one label in $L'$ is valid for $uu_3$.

\noindent Finally, we consider the edge $uu_2$. At this stage, at most six colors appear on edges adjacent to $uu_2$, and hence $f_r(uu_2)\ge 3$. These three colors form three pairwise intersecting candidate labels for $uu_2$. Since $uu_2$ has at most two intermediate vertices, Lemma \ref{l1} ensures that at least one of these labels is valid for $uu_2$.\end{proof}

We are now ready to prove the main results.

\begin{theorem}\label{t5}
Every claw-free subcubic graph admits a $2$-tone edge $11$-coloring.
\end{theorem}

\begin{proof}
Suppose, for a contradiction, that the statement is false, and let $G$ be a counterexample with the minimum number of vertices.

\noindent Our general strategy is to select a vertex $u \in G$ and consider the graph $G' = G - u$. Since the property of being claw-free and subcubic is closed under taking subgraphs, $G'$ is also a claw-free subcubic graph. By the minimality of $G$, the graph $G'$ admits a $2$-tone edge $11$-coloring. By Remark \ref{r1}, this coloring is a partial $2$-tone edge $11$-coloring of $G$ where only the edges incident with $u$ are uncolored. We then demonstrate that this partial coloring can be extended to the uncolored edges, yielding a contradiction. To this end, we first derive several structural properties of $G$.

\noindent Clearly, $G$ is connected; otherwise, by the minimality of $G$, each connected component of $G$ admits a $2$-tone edge $11$-coloring, and the union of these colorings yields a $2$-tone edge $11$-coloring of $G$, a contradiction.

\noindent We next show that $G$ has no vertex of degree one. Suppose, to the contrary, that $u$ is a vertex of degree one in $G$ with neighbor $v$, and let $G' = G - u$. By the minimality of $G$, the graph $G'$ admits a $2$-tone edge $11$-coloring $f$. By Lemma \ref{l3}, $f$ extends to the uncolored edge $uv$, yielding a $2$-tone edge $11$-coloring of $G$, a contradiction.

\noindent Similarly, $G$ has no vertex of degree two. Let $u$ be a vertex of degree two in $G$ with neighbors $v$ and $w$, and let $G' = G - u$. By the minimality of $G$, $G'$ admits a $2$-tone edge $11$-coloring $f$. By Lemma \ref{l4}, $f$ extends to both edges $uv$ and $uw$, again producing a contradiction. Therefore, $G$ is a cubic graph.

\noindent Let $u$ be an arbitrary vertex in $G$ with neighbors $u_1,u_2,u_3$. Since $G$ is claw-free, $\{u_1,u_2,u_3\}$ is not an independent set, and hence at least one pair among these vertices is adjacent. Moreover, $G$ is not isomorphic to $K_4$, since $K_4$ admits a $2$-tone edge $9$-coloring. Therefore, $\{u_1,u_2,u_3\}$ does not induce a triangle in $G$.

\noindent We claim that exactly one pair among $\{u_1,u_2,u_3\}$ is adjacent. Otherwise, suppose that one vertex among $\{u_1,u_2,u_3\}$ is adjacent to the other two; without loss of generality, assume that $u_2$ is adjacent to both $u_1$ and $u_3$. Let $G'=G-u$. By the minimality of $G$, the graph $G'$ admits a $2$-tone edge $11$-coloring $f$. By Lemma~\ref{l5}, $f$ extends to the edges $uu_1$, $uu_2$, and $uu_3$, yielding a contradiction.

\noindent Thus, without loss of generality, we may assume that $u_1u_2\in E(G)$ and that $u_1u_3,u_2u_3\notin E(G)$. Let $G' = G - u$. By the minimality of $G$, $G'$ admits a $2$-tone edge $11$-coloring $f$. Let $u_1'$ and $u_2'$ denote the third neighbors of $u_1$ and $u_2$ in $G$, respectively. We proceed to extend $f$ to $uu_1, uu_2,$ and $uu_3$ in this order.

\noindent Suppose, without loss of generality, that $f(u_1u_2)=12$ and $f(u_1u_1')=34$. Since $d_G(u_1u_1',u_2u_2')\le2$, the label $f(u_2u_2')$ contains a color different from both $3$ and $4$; without loss of generality, assume that it contains $5$. Thus, at least one color from $\{1,2,3,4\}$ is free at $uu_2$; let this color be $3$. Consequently, we assume $f(u_2u_2')$ is either $45$ or $56$.

\noindent It follows that the colors $5,6,7,8,9$ are free at $uu_1$. Consider the family $L=\{56, 57, 58, 59\}$ of pairwise intersecting candidate labels for $uu_1$. Since $\eta_G(uu_1)\le 3$, Corollary~\ref{x} ensures that at least one label in $L$ is valid for $uu_1$. Without loss of generality, we assign $59$ to $uu_1$.

\noindent We next consider the edge $uu_2$. Since the color $5$ appears on both $uu_1$ and $u_2u_2'$, it follows that $f_r(uu_2)\ge 6$. In particular, the colors $3,7,8,10,11$ are free at $uu_2$. Let $L'=\{37,38,310,311\}$ be a family of pairwise intersecting candidate labels for $uu_2$. Let $e$ and $e'$ denote the $uu_3$-edges. By Remark \ref{r2}, at most one of $e$ and $e'$ can be a forbidding edge for $uu_2$ under $L'$.

\noindent Assume, without loss of generality, that $e$ is a forbidding edge for $uu_2$ under $L'$. Since $f(u_1u_1')=34 \notin L'$, and since $uu_2$ has at most two intermediate vertices distinct from $u_1$, Lemma~\ref{l1} implies that at least two labels from $L'$ are valid for $uu_2$. Without loss of generality, assume that these labels are $310$ and $311$. In both labels, $uu_2$ and $e$ share the color $3$. Thus, we have $f_r(uu_3) \ge 4$, yielding at least six candidate labels for $uu_3$.

\noindent Observe that neither $u_1u_1'$ nor $u_2u_2'$ forbids any of these candidate labels for $uu_3$, since $f(u_1u_1')$ contains the color $5$ and $f(u_2u_2')$ contains the color $3$, neither of which appears among the free colors at $uu_3$. Therefore, $uu_3$ has at most five forbidding edges, and hence at least one of its candidate labels is valid. Thus, $f$ extends to the three edges, a contradiction.

\noindent Thus, neither $e$ nor $e'$ forbids a label from $L'$ for $uu_2$. It follows that at least three labels in $L'$ are valid for $uu_2$; assume these are $38, 310,$ and $311$. If any of the colors $3,8,10,11$ appears on $f(e)$ or $f(e')$, then again $f_r(uu_3)\ge 4$, yielding at least six candidate labels for $uu_3$, and as before, $uu_3$ has at most five forbidding edges. Therefore, at least one of these labels is valid, a contradiction. Hence, none of these colors appears on $f(e)$ or $f(e')$.

Now assign $38$ to $uu_2$. Then at least three colors are free at $uu_3$, two of which are $10$ and $11$; let the third be $a$. Let $L'' = \{a10, a11, 1011\}$. Observe that none of the labels of $u_1u_1'$, $u_1u_2$, and $u_2u_2'$ belong to $L''$. Moreover, $uu_3$ has at most two intermediate vertices distinct from $u_1$ and $u_2$. Hence, by Lemma \ref{l1}, at least one label in $L''$ is valid for $uu_3$, and hence $f$ extends to  $uu_1, uu_2,$ and $uu_3$, a contradiction. Thus, such a graph $G$ does not exist.\end{proof}

\begin{theorem}
Every $2$-degenerate subcubic graph is $2$-tone edge $10$-colorable.
\end{theorem}

\begin{proof}
Suppose, for a contradiction, that the theorem is false, and let $G$ be a minimum counterexample. Since $G$ is $2$-degenerate, there exists a vertex $u\in V(G)$ with $d_G(u)\le2$. Let $v$ and $w$ denote the neighbors of $u$ in $G$, and set $G'=G-u$. Observe that $G'$ is also $2$-degenerate and subcubic. By the minimality of $G$, the graph $G'$ admits a $2$-tone edge $10$-coloring $f$. By Remark \ref{r1}, $f$ is a partial $2$-tone edge $10$-coloring of $G$. Then, by Lemma~\ref{l4}, $f$ extends to the uncolored edges $uv$ and $uw$, yielding a contradiction.\end{proof}

We conjecture the following.

\begin{conjecture}\label{c1}
If $G$ is a cubic graph, then $\tau'_2(G)\le 9$.
\end{conjecture}

The bound in Conjecture \ref{c1} is the best possible, since $\tau'_2(K_4)=9$. Moreover, we believe that $K_4$ is the unique graph requiring nine colors. Motivated by this observation, we further propose the following stronger conjecture.

\begin{conjecture}\label{c2}
If $G$ is a cubic graph that does not contain $K_4$, then $\tau'_2(G)\le 8$.
\end{conjecture}

The bound is sharp, since $\tau'_2(K_4-e)=8$. A natural follow-up question is whether excluding $K_4-e$ as a subgraph allows a smaller bound. Computational experiments suggest otherwise: the graph in Figure~\ref{f2} is the smallest $K_4-e$-free graph satisfying $\tau'_2(G)=8$.

\begin{figure}[ht]
\centering
\begin{tikzpicture}[
    scale=0.6,
    every node/.style={
        circle,
        draw=black,
        fill=white,
        thick,
        minimum size=2mm,
        inner sep=0pt
    },
    every path/.style={black, thick}
]

\node (8)  at (0,4)   {};
\node (4)  at (3.4,3.8) {};
\node (1)  at (6.3,2.3) {};
\node (2)  at (5.2,-1.4) {};
\node (3)  at (4.8,0.7) {};
\node (6)  at (4.4,2.1) {};
\node (7)  at (1.8,0.2) {};
\node (5)  at (1.7,-1.6) {};
\node (9)  at (-0.5,0.5) {};
\node (10) at (0.8,2.1) {};

\draw (8) -- (4);
\draw (4) -- (1);
\draw (1) -- (2);
\draw (2) -- (3);
\draw (3) -- (1);
\draw (3) -- (6);
\draw (6) -- (10);
\draw (10) -- (8);
\draw (8) -- (9);
\draw (9) -- (10);
\draw (9) -- (5);
\draw (5) -- (2);
\draw (5) -- (7);
\draw (7) -- (4);
\draw (7) -- (6);

\end{tikzpicture}
\caption{The smallest $K_4-e$-free graph with $2$-tone chromatic index $8$.}
\label{f2}
\end{figure}
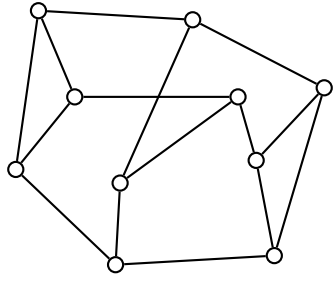

We now prove that every series-parallel subcubic multigraph of order $n \ge 2$ satisfies Conjecture \ref{c1}. Such graphs can be constructed from a single vertex with a loop by repeatedly applying the following two operations: \emph{(i) series operations}, which subdivide an edge and insert a vertex, and  \emph{(ii) parallel operations}, which add a new edge in parallel with an existing one.  

This construction guarantees that the resulting graph has minimum degree at least two. Moreover, since a loop contributes two to the degree of its incident vertex, any series-parallel subcubic multigraph with at least two vertices must be loopless.

\begin{theorem}
Every series-parallel subcubic multigraph of order $n \ge 2$ admits a $2$-tone edge $9$-coloring.
\end{theorem}

\begin{proof}
We proceed by induction on the number of series and parallel operations. If the graph consists of exactly two parallel edges, then it clearly admits the required coloring.

\noindent Now assume that at some stage we have a series-parallel subcubic multigraph $G$ equipped with a valid $2$-tone edge $9$-coloring $f$. Let $G'$ be a graph obtained from $G$ by performing either a series or a parallel operation. We show that $G'$ also admits the required coloring. Note that $f$ is a partial $2$-tone edge $9$-coloring of $G'$.

\noindent First, consider the case where a parallel edge $e$ is added to an existing edge $uv$ of $G$. At least three colors are free at $e$, yielding three pairwise intersecting candidate labels. Since $\eta_G(e)\le 2$, Corollary~\ref{x} ensures that $f$ extends to $e$.

\noindent Now assume that $G'$ is obtained from $G$ by subdividing an edge $uv$ with a new vertex $x$. Assume, without loss of generality, that $f(uv)=12$. Then both colors $1$ and $2$ are free at $ux$ and $xv$.

\noindent First, consider the edge $ux$. Since $xv$ is not yet colored, at least three colors distinct from $1$ and $2$ are free at $ux$. Suppose, without loss of generality, that these colors are $7,8,9$. Then $78$, $79$, and $89$ are candidate labels for $ux$. If at least one of these labels does not appear on a second neighbor of $ux$, then we assign this label to $ux$ and then assign $12$ to $xv$, obtaining a $2$-tone edge $9$-coloring of $G'$.

\noindent Thus, we may assume that each of the labels $78$, $79$, and $89$ appears on a second neighbor of $ux$. Hence, there exist two distinct neighbors $u_1$ and $u_2$ of $u$ distinct from $x$, and a neighbor $v_1$ of $v$ distinct from $x$, such that $vv_1$ is assigned one of the labels in $\{78,79,89\}$ and the remaining two labels are then assigned to a $u_1u$-edge and a $u_2u$-edge, respectively. Without loss of generality, suppose that $f(vv_1)=89$.

\noindent We aim to assign $ux$ a label containing $8$ or $9$. Let $L=\{18, 19, 28, 29\}$ be the set of candidate labels for $ux$. No $xv$-edge forbids a label from $L$ for $ux$, while at most one $uu_1$-edge and at most one $uu_2$-edge can forbid a label from $L$ for $ux$. Consequently, at most two second neighbors of $ux$ may carry labels from $L$. Thus, at least one of these labels can be assigned to $ux$; without loss of generality, we assign $28$ to $ux$.

\noindent Finally, consider the edge $xv$. At least three colors other than $1$ are free at $xv$; denote them by $a$, $b$, and $c$. Consider $L=\{a1,b1,c1\}$. No $xu$-edge forbids labels from $L$ for $xv$. Moreover, $xv$ has at most two intermediate vertices distinct from $u$. Hence, by Lemma \ref{l1}, at least one label in $L$ is valid for $xv$.\end{proof}

Since every series-parallel subcubic multigraph is $K_4$-minor-free, and hence $K_4$-free, Conjecture \ref{c2}, if true, would imply that this bound is not sharp.

We now prove that every subcubic outerplanar graph satisfies Conjecture \ref{c2}. An outerplanar graph is a graph that admits a planar embedding in which every vertex lies on the boundary of the outer face. In such an embedding, faces enclosed by a finite cycle are called bounded faces.

Given an outerplane graph $G$, the \emph{weak dual} $\tau(G)$ is the graph whose vertices correspond to the bounded faces of $G$, with two vertices adjacent whenever the corresponding faces share an edge. For a bounded face $F$ of $G$, we denote by $C(F)$ the cycle of $G$ that bounds $F$. It is well known that a plane graph is outerplane if and only if its weak dual is a forest. Let $G$ be a subcubic outerplane graph. A bounded face $F$ of $G$ is called a \emph{pendant} face if $F$ corresponds to a leaf of $\tau(G)$ and $C(F)$ contains either exactly one vertex of degree three or exactly two consecutive vertices of degree three. 

\begin{theorem}\label{t7}
Every subcubic outerplanar graph $G$ admits a $2$-tone edge $8$-coloring.
\end{theorem}

\begin{proof} 
Suppose, to the contrary, that the theorem is false. Let $G$ be a counterexample with the minimum number of vertices. Clearly, $G$ is connected. Moreover, $G$ has no vertex of degree one. Otherwise, suppose that $u$ is a vertex of degree one in $G$ with neighbor $v$, and let $G' = G - u$. By the minimality of $G$, $G'$ admits a $2$-tone edge $8$-coloring $f$. By Lemma \ref{l3}, $f$ extends to the uncolored edge $uv$, yielding a $2$-tone edge $8$-coloring of $G$, a contradiction.

\noindent The graph $G$ must contain at least two bounded faces; otherwise, $G$ is a cycle, which admits a $2$-tone edge $8$-coloring by Corollary \ref{cc1}, a contradiction. Let $F$ be a pendant face in $G$ with $C(F)=v_1v_2\cdots v_\ell$, where $\ell\ge 3$.

\begin{claim}
$\ell= 3$.
\end{claim}

\begin{proof}
Suppose, for the sake of contradiction, that $\ell \ge 4$. By the definition of a pendant face, we may assume, without loss of generality, that either both $v_1$ and $v_\ell$ have degree three in $G$, or only $v_1$ has degree three. Let $G' = G - \{v_3\}$. By the minimality of $G$, the graph $G'$ admits a $2$-tone edge $8$-coloring $f$, which is a partial coloring of $G$ by Remark \ref{r1}. 

\noindent We sequentially color the edges in $\{v_2v_3,v_3v_4\}$, beginning with $v_3v_4$ when $\ell=4$. For each chosen edge $e$, at least four colors are free at $e$, yielding at least six candidate labels, while $e$ has at most three second neighbors. Hence, $f$ extends to both edges, producing a $2$-tone edge $8$-coloring of $G$, a contradiction.\end{proof}

\noindent Thus, $F$ is a triangle. Suppose first, without loss of generality, that $d_G(v_1)=3$ and $d_G(v_2)=d_G(v_3)=2$. Let $G'=G-\{v_2,v_3\}$. By the minimality of $G$, $G'$ admits a $2$-tone edge $8$-coloring $f$. Let $y$ be the neighbor of $v_1$ in $G'$, and let $y_1$ be a neighbor of $y$ in $G$. The vertex $y$ has at most one additional neighbor; denote it (if it exists) by $y_2$. Without loss of generality, assume that $f(v_1y)=12$, $f(yy_1)=34$, and $f(yy_2)=56$ (if $y_2$ exists). Assign the labels $47$, $16$, and $35$ to the edges $v_1v_2$, $v_2v_3$, and $v_1v_3$, respectively. This yields a $2$-tone edge $8$-coloring of $G$, a contradiction.

\noindent Therefore, we may assume that $d_G(v_1)=d_G(v_2)=3$ and $d_G(v_3)=2$. Let $G'=G-\{v_3\}$. By the minimality of $G$, $G'$ admits a $2$-tone edge $8$-coloring $f$. By Remark \ref{r1}, $f$ is a partial coloring of $G$. Let $y_1$ and $y_2$ be the neighbors of $v_1$ and $v_2$, respectively, outside $C(F)$. Without loss of generality, assume that $f(v_1v_2)=12$ and $f(v_1y_1)=34$. Since $d_G(v_1y_1, v_2y_2)\le 2$, the label $f(v_2y_2)$ contains a color distinct from both $3$ and $4$; without loss of generality, assume it contains $5$.

\noindent We aim to assign to the edge $v_1v_3$ a label containing the color $5$. Observe that the colors $5,6,7,8$ are free at $v_1v_3$. Among the candidate labels for $v_1v_3$, consider the pairwise intersecting family $L=\{56,57,58\}$. Since $\eta_G(v_1v_3)\le 2$, it follows from Corollary \ref{x} that one of the labels in $L$ is valid for $v_1v_3$.

\noindent Finally, consider the edge $v_2v_3$. At least three colors are free at $v_2v_3$, yielding three candidate labels that are pairwise intersecting. Since $\eta_G(v_2v_3)\le 2$, it follows from Corollary \ref{x} that one of these candidate labels is valid for $v_2v_3$, a contradiction.\end{proof}

For subcubic outerplanar graphs $G$, Theorem \ref{t4} gives $\tau'_2(G)\le 14$, whereas Theorem \ref{t7} reduces this bound to $8$.

\paragraph{Acknowledgment.}
The author would like to thank Prof. Olivier Togni for his encouragement and support throughout this work. The author is also grateful to him for carefully reading the manuscript, providing valuable remarks, and carrying out the computational experiments leading to the counterexample in Figure \ref{f2}, including the verification that no smaller counterexample exists.


\begin{thebibliography}{99}

\bibitem{B}
J. Balogh, M. Kochol, A. Pluh\'{a}r, X. Yu,
Covering planar graphs with forests,
J. Combin. Theory Ser. B 94 (2005) 147--158. https://doi.org/10.1016/j.jctb.2004.12.002.

\bibitem{B1}
A. Bickle,
$2$-tone coloring of joins and products of graphs,
Congr. Numer. 217 (2013) 171--190.

\bibitem{BP}
A. Bickle, B. Phillips,
$t$-tone colorings of graphs,
Util. Math. 106 (2018) 85--102.

\bibitem{C1}
D.W. Cranston, J. Kim, W.B. Kinnersley,
New results in $t$-tone coloring of graphs,
Electron. J. Combin. 20 (2013), P17.

\bibitem{Dong1}
J. Dong,
$2$-tone coloring of graphs with maximum degree 4,
Util. Math. 107 (2018) 19--22.

\bibitem{E}
P. Erd\H{o}s, C. Ko, R. Rado,
Intersection theorems for systems of finite sets,
Quart. J. Math. Oxford Ser. (2) 12 (1961) 313--320.

\bibitem{FA}
I. Fabrici,
Light graphs in families of outerplanar graphs,
Discrete Math. 307 (2007) 866--872. https://doi.org/10.1016/j.disc.2005.11.058.


\bibitem{F}
N. Fonger, J. Goss, B. Phillips, C. Segroves,
Math 6450: Final Report,
Unpublished manuscript, 2009.
Available at:
\url{https://homepages.wmich.edu/zhang/finalReport2.pdf}

\bibitem{Kang}
R.J. Kang, P. Manggala,
Distance edge-colourings and matchings,
Discrete Appl. Math. 160 (2012) 2435--2439.

\bibitem{KS}
F. Kramer, H. Kramer,
A survey on the distance coloring of graphs,
Discrete Math. 308 (2008) 422--426.
\end{thebibliography}
\end{document}